\documentclass[12 pt,oneside]{amsart}

\usepackage{mathabx}

\usepackage{amsmath}
\usepackage{amsxtra}
\usepackage{amscd}
\usepackage{amsthm}
\usepackage{amsfonts}
\usepackage{amssymb}
\usepackage{eucal}
\usepackage{graphicx}
\usepackage{lipsum}
\graphicspath{ {images/} }

\usepackage[margin=1.1 in]{geometry}
\usepackage{marginnote}

\theoremstyle{definition}
\newtheorem{thm}{Theorem}[section]

\newtheorem{lemma}[thm]{Lemma}

\newtheorem{claim}[thm]{Claim}
\newtheorem{condition}[thm]{Condition}

\newtheorem*{example}{Example}

\DeclareFontFamily{U}{mathx}{}
\DeclareFontShape{U}{mathx}{m}{n}{<-> mathx10}{}
\DeclareSymbolFont{mathx}{U}{mathx}{m}{n}
\DeclareMathAccent{\widehat}{0}{mathx}{"70}
\DeclareMathAccent{\widecheck}{0}{mathx}{"71}

\newcommand\nc{\newcommand}

\nc\Span{\text{\rm Span}}
\nc\Id{\text{Id}}

\nc \cc {\mathbb{C}}
\nc \dd {\mathbb{D}}
\nc \ff {\mathbb{F}}
\nc \hh {\mathbb{H}}
\nc \ii {\mathbb{I}}
\nc \nn {\mathbb{N}}
\nc \oo {\mathbb{O}}
\nc \pp {\mathbb{P}}
\nc \qq {\mathbb{Q}}
\nc \rr {\mathbb{R}}

\nc \zz {\mathbb{Z}}

\nc\cA{\mathcal{A}}
\nc\cB{\mathcal{B}}
\nc\cC{\mathcal{C}}
\nc\cD{\mathcal{D}}
\nc\cE{\mathcal{E}}
\nc\cF{\mathcal{F}}
\nc\cG{\mathcal{G}}
\nc\cH{\mathcal{H}}
\nc\cK{\mathcal{K}}
\nc\cL{\mathcal{L}}
\nc\cM{\mathcal{M}}
\nc\cN{\mathcal{N}}
\nc\cO{\mathcal{O}}
\nc\cP{\mathcal{P}}
\nc\cQ{\mathcal{Q}}
\nc\cR{\mathcal{R}}
\nc\cS{\mathcal{S}}
\nc\csf{\mathcal{S}\mathcal{F}}

\nc\fs{\mathfrak{s}}
\nc\fg{\mathfrak{g}}
\nc\fm{\mathfrak{m}}
\nc\fp{\mathfrak{p}}
\nc\fso{\mathfrak{so}}
\nc\fsu{\mathfrak{su}}

\nc\Tr{\text{Tr}}
\nc\into{\hookrightarrow}

\nc\st{\text{ s.t. }}
\nc\intense[1]{\textcolor[rgb]{1.00,0.00,0.00}{\textbf{#1}}}

\renewcommand{\(}{\left(}
\renewcommand{\)}{\right)}

\nc\Mat{\text{\rm Mat}}
\nc\GL{\text{\rm GL}}
\nc\SU{\text{\rm SU}}
\nc\SO{\text{\rm SO}}
\nc\SL{\text{\rm SL}}
\nc\Sp{\text{\rm Sp}}
\nc\EL{\text{\rm EL}}

\nc\GEM{\text{\rm GEM}}
\nc\Alt{\text{\rm Alt}}
\nc\Sym{\text{\rm Sym}}
\nc\Hess{\text{Hess}}

\nc\Crit{\text{Crit}}

\renewcommand{\(}{\left(}
\renewcommand{\)}{\right)}

\nc\inject{\hookrightarrow}

\nc{\mattwo}[4]{\left[\begin{array}{cc} #1  & #2\\  #3 & #4 \\ \end{array} \right]}
\nc{\matthree}[9]{\left[\begin{array}{ccc} #1  & #2 & #3\\  #4 & #5 & #6 \\ #7 & #8 & #9 \\ \end{array} \right]}
\nc{\vecttwo}[2]{\left[\begin{array}{c} #1 \\ #2 \\ \end{array} \right]}
\nc{\vectthree}[3]{\left[\begin{array}{c} #1 \\ #2\\  #3 \\ \end{array} \right]}

\nc{\del}{\partial}
\nc\onto{\twoheadrightarrow}

\nc\const{\text{const}}

\nc\rrp{\rr P}
\nc\ul{\underline}
\nc\ol{\overline}
\nc\uline{\underline}
\nc\oline{\overline}
\nc\oset{\overset}
\nc\uset{\underset}

\nc\heart{\heartsuit}
\nc\spade{\spadesuit}
\nc\club{\clubsuite}

\nc\marg[1]{\marginnote{\boxed{\text{#1}}}}
\nc\margq[1]{\marginnote{\textcolor[rgb]{1.00,0.00,0.00}{#1}}}

\nc\Char{\text{Char}}
\nc\Frac{\text{Frac}}
\nc\wo{\backslash}
\nc\diag{\text{diag}}
\nc\wtl{\widetilde}
\nc\nsubgp{\triangleleft}

\nc\Cay{\text{Cay}}
\nc\Hom{\text{Hom}}
\nc\Gp{\text{Gp}}
\nc\Set{\text{Set}}
\nc\la{\langle}
\nc\ra{\rangle}

\nc\Spec{\text{Spec}}
\nc\ad{\text{ad}}

\nc\wht{\widehat}
\nc\ddx[2]{\frac{\partial {#1}}{\partial {#2}}}
\nc\dddx[3]{\frac{\partial^2 {#1}}{\partial {#2}\partial{#3}}}
\nc\mult{\text{mult}}
\nc\supp{\text{supp}}
\nc\sign{\text{sign}}

\nc\tr{\text{tr}}
\nc\stab{\text{stab}}
\nc\im{\text{im}}

\nc\sech{\text{sech}}

\nc\Ind{\text{Ind}}
\nc\tsf{\text{sf}}
\nc\End{\text{End}}
\nc\Hol{\text{Hol}}
\nc\hol{\text{hol}}
\nc\Lie{\text{Lie}}

\nc\vol{\text{vol}}
\nc\ind{\text{ind}}
\nc\Met{\mathcal{M}\text{et}}
\nc\Grass{\mathcal{G}\text{rass}}

\nc\longto{\longrightarrow}
\nc\grad{\text{grad}}
\nc\Map{\text{Map}}
\nc\Indx{\text{Ind}}
\nc\indx{\text{ind}}
\nc\pt{\text{pt}}
\nc\Aut{\text{Aut}}
\nc\Br{\text{Br}}
\nc\Gr{\text{Gr}}
\nc\CS{\text{CS}}

\nc\Proj{\text{Proj}}
\nc\Ad{\text{Ad}}
\nc\Diff{\text{Diff}}
\nc\sing{\text{sing}}
\nc\order{\text{order}}
\nc\bsl{\backslash}

\nc\BP{\textbf{P}}
\nc\red{\text{red}}
\nc\Kh{\text{Kh}}
\nc\Tor{\text{Tor}}
\nc\spinc{\text{spin}^c}

\newcommand\blfootnote[1]{%
  \begingroup
  \renewcommand\thefootnote{}\footnote{#1}%
  \addtocounter{footnote}{-1}%
  \endgroup
}

\title[Families of p.s.c. metrics on spectral sequence cobordisms]{Families of metrics with positive scalar curvature on spectral sequence cobordisms}
\author{Sherry Gong}

\begin{document}

\begin{abstract}

We study families of metrics on the cobordisms that underlie the differential maps in Bloom's monopole Floer spectral sequence, a spectral sequence for links in $S^3$ whose $E^2$ page is the Khovanov homology of the link, and which abuts to the monopole Floer homology of the double branched cover of the link.

The higher differentials in the spectral sequence count parametrized moduli spaces of solutions to Seiberg-Witten equations, parametrized over a family of metrics with asymptotic behaviour corresponding to a configuration of unlinks with 1-handle attachments. For a class of configurations, we construct families of metrics with the prescribed behaviour, such that each metric therein has positive scalar curvature. The positive scalar curvature implies that there are no irreducible solutions to the Seiberg-Witten equations and thus, when the spectral sequences are computed with these families of metrics, only reducible solutions must be counted.

The class of configurations for which we construct these families of metrics includes all configurations that go into the spectral sequence for $T(2,n)$ torus knots, and all configurations that involve exactly two 1-handle attachments.

\end{abstract}

\maketitle

\section{Introduction}

Monopole Floer homology is a gauge-theoretic invariant for 3-manifolds defined via a process analogous to Morse homology, using the Chern-Simons-Dirac functional. Its underlying chain complex is generated by Seiberg-Witten monopoles on the 3-manifold, and differentials count monopoles over the product of the 3-manifold with $\rr$.\blfootnote{Author supported by US National Science Foundation grant DMS-2055736.}

In \cite{Bloom}, Bloom constructed a spectral sequence associated to a projection of a knot or link, from the reduced Khovanov homology of the link to a version of monopole Floer homology of the branched double cover of the link. The differentials in the spectral sequence arise from counting solutions to the Seiberg-Witten equations over a family of metrics on the cobordism.

There are two kinds of solutions to the Seiberg-Witten equations that come into the
differentials: the irreducible ones and the reducible ones. If the family of metrics on the cobordism
can be chosen such that all of the metrics have positive scalar curvature, this would mean that
there are no irreducible solutions, and the monopole Floer homology would be computable by only counting reducible solutions.

This effect has been studied in the reverse direction of using invariants related to the monopole Floer homology to obstruct metrics of positive scalar curvature: in \cite{SeibWit_psc}, Seiberg and Witten used an invariant coming from counting solutions to Seiberg-Witten equations to obstruct the existence of a metric of positive scalar curvature on closed $\spinc$ $4$-manifolds $X$ with $b_2^+(X)>0$.  In \cite{Ruberman2001PositiveSC}, Ruberman used Seiberg-Witten invariants to give examples of simply connected $4$-manifolds for which the space of metrics of positive scalar curvature is disconnected. In \cite{Lin_psc_end_periodic}, Lin gave a new obstruction to the existence of positive scalar curvature metrics on compact $4$ manifolds with the same homology as $S^1 \times S^3$ by studying Seiberg-Witten equations.

The purpose of this paper is to construct families of metrics of positive scalar curvature for cobordisms arising in Bloom's spectral sequence for certain classes of link diagrams, and in doing so give a better understanding of the monopole Floer homology and Bloom's spectral sequence relating it to Khovanov homology.

To specify the particular classes of link diagrams, consider a link projection $P_L \subset S^2$ with the black-and-white checkerboard colouring, as in Figure \ref{checkerboard}. 

\begin{figure}[ht!]
\centering
\includegraphics[width=60mm]{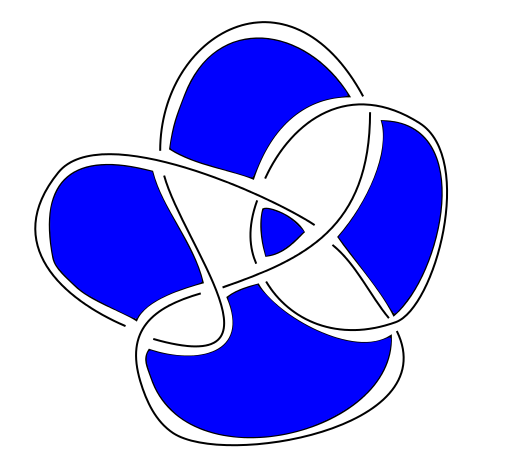}
\caption{The checkerboard colouring for the link $L8N1$. (The blue parts are the ``black regions''.)}
\label{checkerboard}
\end{figure}

Let $G$ be the graph associated with $P_L$, with vertices corresponding to black regions of $P_L$, and edges corresponding to crossings; for the link projection depicted in Figure \ref{checkerboard}, $G$ has five vertices and eight edges.

Let $L_0$ be the unlink obtained by resolving all the crossings in $P_L$ in a way such that each black region becomes a component, and let $L_1$ be the resolutions that reverses all the crossings of $L_0$. In particular, if we were considering an alternating link projection, these would be the $0$ and $1$ resolutions.

For the link projection depicted in Figure \ref{checkerboard}, the unlink $L_0$ can be seen as the five circles that form the boundary of the five blue regions. The information of the link projection can now be expressed as an unlink $L_0$ along with some crossings, as shown in Figure \ref{loops_with_crossings}. In order fully to capture the information of the link, each crossing is coloured in one of two colours representing whether the resolution in $L_0$ is the $0$ or $1$ resolution.

\begin{figure}[ht!]
\centering
\includegraphics[width=60mm]{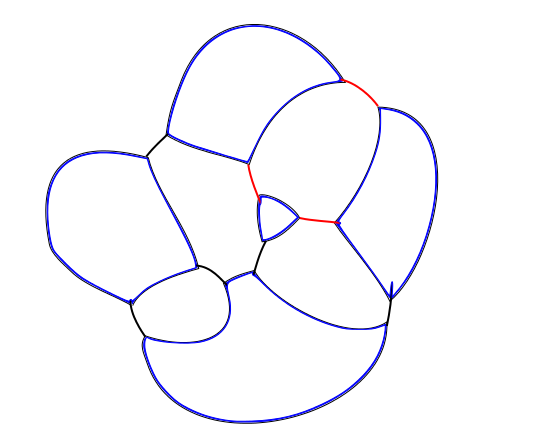}
\caption[]{$L_0$ and the crossings for the $L8N1$ projection above in Figure \ref{checkerboard}. The red and black crossings are for when the resolution into $L_0$ involved the $0$ resolution and the $1$ resolution, respectively.)}
\label{loops_with_crossings}
\end{figure}

The main theorem of this paper is the following.

\begin{thm} Suppose that the black graph $G$ has a vertex such that all of the edges have one end (but not both ends) on that vertex. Then there is a family of metrics on $W(IJ)$ for any pair of resolutions on $I$ and $J$ of $P_L$, parametrising the translations of the handles, as described below, such that every metric in the family has positive scalar curvature.
\label{psc_family_thm} 
\end{thm}

Note that in our definition, we are allowed to reverse black and white, and the conditions of the lemma are not symmetric in this reversal. For example, for a trefoil, for one choice of colouring of black and white, the black graph consists of two vertices and three edges all of which go between the vertices, which satisfies the condition of the lemma. For the other choice, the graph has three vertices, $x,y,z$ and three edges, one between each pair of $x$, $y$, and $z$, and does not satisfy the conditions.

A consequence of this theorem is that for these links, the differentials corresponding to cube-faces in Bloom's spectral sequence can be computed by counting only the reducible solutions to the Seiberg-Witten equations, since the positivity of the scalar curvature of the family of metrics that solutions are being counted on precludes the existence of irreducible solutions.

In \cite{szabo_geometric} Szab\'{o} defined a combinatorial spectral sequence meant to model a spectral sequence for Heegaard Floer homology. One may hope to use a count of reducible solutions to construct an analogous combinatorial spectral sequence for monopole Floer homology.

\begin{example}
Some examples of link diagrams for which the theorem may be applied are minimal projections of the torus links $T_{2,n}$ and connected sums of such. For these, if we take the standard projection of $T_{2,n}$s, connect-sum them along their wings, and then take the black-and-white checkerboard colouring with the outside region coloured black, as in Figure \ref{t2n_conn_sums}, it is easy to see that all edges will emanate from the vertex corresponding to the outside region.

\begin{figure}[ht!]
\centering
\includegraphics[width=80mm]{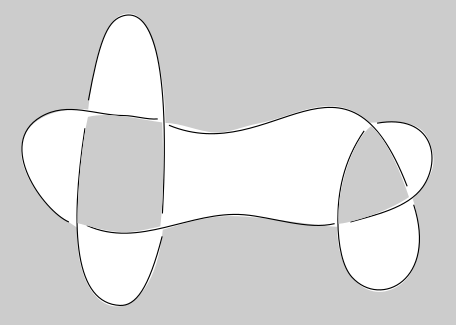}
\caption{A valid black-and-white checkerboard colouring for a connected sum of $T_{2,4}$ and $T_{2,3}$.}
\label{t2n_conn_sums}
\end{figure}
\end{example}

The four-dimensional cobordisms in \cite{Bloom} are topologically described as cobordisms between connected sums of $S^1 \times S^2$s given by the attachment of $2$-handles ($D^2 \times D^2$s) that are attached along $S^1 \times D^2$s that are away from each other. These connected sums of $S^1 \times S^2$s arise as double-branched covers of $S^3$ along unlinks, where the unlinks come from resolving the crossings of the original link projection. The $2$-handle attachments to the connected sums of $S^1 \times S^2$s arise as the double branched cover of $1$-handles attached to the unlinks, where the $1$-handles are the same 1-handles that underlie the merge and split maps between the $0$- and $1$-resolutions of the crossings.

The families of metrics of positive scalar curvature we construct are such that away from the attaching $S^1 \times D^2$s, where the cobordism looks like a $(\#^k(S^1 \times S^2)) \times \rr$, the metrics are constant in the $\rr$ axis and given by taking standard metrics on the $S^1 \times S^2$s away from the points at which they are connected-summed to each other.

In the notation of configurations of unlinks with $1$-handle attachments, this means that reversing the sign of a crossing does not affect whether we can construct our families of metrics, since we can easily time-reverse our construction withing the active region of the handle-attachments. This is why the signs of the crossings do not come into the statement of the theorem.

In fact, the theorem may be stated more generally as follows:

\begin{thm}\label{thm_in_terms_of_L0}
Suppose there is a resolution $L_0$ of the link projection $P_L$ (with each crossing resolved with either the $0$ or $1$ resolutions, but they do not all have to be resolved the same way as each other), with unlink components $c_0, c_1, \ldots c_n$ such that all the handle attachments involved in going from the resolution $L_0$ to its opposite resolution $L_1$ are between $c_0$ and $c_i$ for $i \geq 1$. 

Then there is a family of metrics on $W(IJ)$ for any pair of resolutions on $I$ and $J$ of $P_L$, parametrising the translations of the handles, such that every metric in the family has positive scalar curvature.
\end{thm}

\begin{example} For any configuration with exactly two handle attachments, that is for any configuration corresponding to a link projection with two crossings, the theorem applies. These are the configurations Szab\'{o} calls \textit{2-dimensional configurations} in \cite{szabo_geometric}.

To see this, observe that all such configurations arise as resolutions of link projections with two crossings, of which there are only four types, depicted in Figure \ref{dimension2_configs}.

\begin{figure}[ht!]
\centering
\includegraphics[width=120mm]{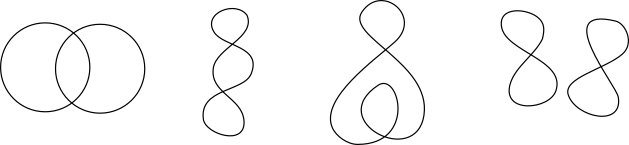}
\caption{Configurations with exactly two crossings. The signs of the crossings are not shown because they are irrelevant to the application of the theorem.}
\label{dimension2_configs}
\end{figure}

Resolutions $L_0$ for the first three configurations satisfying the conditions in Theorem \ref{thm_in_terms_of_L0} are depicted in Figure \ref{dim2_configs_resolved}, with the circle $c_0$ shown in blue, and the relevant $1$-handle attachments depicted by the green arcs.

\begin{figure}[ht!]
\centering
\includegraphics[width=100mm]{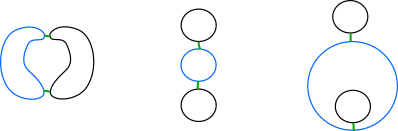}
\caption{Resolutions for the first three configurations in Figure \ref{dimension2_configs} with a single circle (the blue one) such that all handle attachments have exactly one end on it.}
\label{dim2_configs_resolved}
\end{figure}

For the last configuration, the one with two figure eights, can be viewed as a partial resolution of $T_{2,4}$, with the two crossings marked with blue dots in Figure \ref{T2n_marked} already resolved. Thus, the cobordisms arising for this configuration also arise in the cube for $T_{2,4}$, and as explained in the previous example, satisfy the conditions of the theorem.

\begin{figure}[ht!]
\centering
\includegraphics[width=20mm]{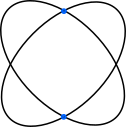}
\caption{}
\label{T2n_marked}
\end{figure}

\end{example}

\vspace{0.5 cm}

\noindent\textbf{Acknowledgements} This project originally started as part of my PhD thesis under the supervision of Tom Mrowka, and I would like to thank him for suggesting the topic and for insights, discussion, and encouragement. I would like to thank Jianfeng Lin for helpful discussions and encouragement, and Mark Chilenski for helpful suggestions of computational tools and libraries for studying differential inequalities.

\section{Background}

All of our monopole Floer homology and Khovanov homology groups will be over the field of two elements, $\ff_2$. 

In \cite{Bloom}, Bloom constructed a spectral sequence relating the Khovanov homology of a link in $S^3$ to the monopole Floer homology of the double branched cover of the link, thus:

For a projection $P_L$ of the link with $k$ crossings, consider its resolutions $L_I$ for $I \in \{0,1\}^k$. Let $Y$ be the double branched cover of $L \subset S^3$, and let $Y(I)$ be the double branched cover of $L_I \subset S^3$. Let $\check{C}$ denote the complex for the ``to'' version of the monopole Floer chain complex, as in the notation of \cite{KM_monopole_book}. For the complex $\wtl{C}$ given by the cone of the map $U_\dagger$ on $\check{C}$,  Bloom constructed the filtered complex
\[X = \oplus_{I \in \{0,1\}^l} \wtl{C}(Y(I))\]
with differential $\wtl{D}$ given by $\wtl{D}^I_J:\wtl{C}(Y(I)) \to \wtl{C}(Y(J))$ for $I \leq J$. The aforementioned spectral sequence is the associated spectral sequence of this complex; Bloom showed that the homology of its $E_1,d_1$ page is the reduced Khovanov homology of the link, and the spectral sequence abuts to $\wtl{HM}(Y)$, where $\wtl{HM}$ is a version of monopole Floer homology that Bloom constructed; it is the homology associated to $\wtl{C}$.

The differentials $\wtl{D}^I_J$ in the spectral sequence arise from counting monopoles on a cobordism $W(IJ)$ between $Y(I)$ and $Y(J)$, over a family of metrics on $W(IJ)$ of dimension $k-1$ where $I$ and $J$ differ at $k$ crossings.

To see this family of metrics, let us first describe the cobordism $W(IJ)$. Note that $L_I$ and $L_J$, being resolutions of $L$, are contained in $S^2$, which we view in $S^3$ as an equatorial $S^2$. Moreover, $L_J$ is obtained from $L_I$ by taking $k$ small disjoint 3-balls $B$ where $L_I \cap B$ looks like the left hand side of Figure \ref{resolutions} and replacing the interiors of these balls with the right hand side of Figure \ref{resolutions} yields $L_J$. In particular, $I$ and $J$ both intersect each of the $3$-balls $B$ in two arcs, and their intersection with $\partial B$ is the same four points. 

\begin{figure}[ht!]
\centering
\includegraphics[width=60mm]{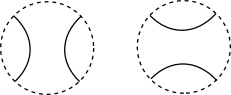}
\caption{The $0$ and $1$ resolutions of a crossing.}
\label{resolutions}
\end{figure}

Then there is a cobordism $S_{IJ} \subset S^3 \times \rr$ from $L_I$ to $L_J$ given by attaching the natural 1-band for each crossing where the resolutions differ, as in Figure \ref{one_saddle}.

\begin{figure}[ht!]
\centering
\includegraphics[width=60mm]{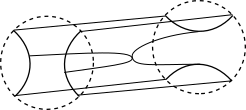}
\caption{The cobordism $S_{IJ}$ near a crossing where the $I$ and $J$ resolutions differ.}
\label{one_saddle}
\end{figure}

The cobordism $W(IJ)$ is the double branched cover of $S^3 \times \rr$ along $S_{IJ}$. This has a natural map to $\rr$, where the pre-image of $t \in \rr$ is the double branched cover of $S^3 \times \{t\}$ along $S_{IJ} \cap S^3 \times \{t\}$.

If we think of $S_{IJ}$ as a movie in $S^3$ from $L_I$ to $L_J$, the action all happens within $k$ disjoint balls $B_i$ around the $k$ crossings where the $I$ and $J$ resolutions differ. Thus, we may consider for each tuple $(t_1, \ldots, t_k)$ with $t_i \in \rr$, a cobordism $S_{IJ, (t_1,\ldots ,t_k)} \subset S^3 \times \rr$ where outside of the $B_i$, all $S_{IJ,(t_1, \ldots, t_k)}$ are the same and are constant in the $\rr$ factor, and the saddle point in the ball $B_i$ occurs at $t_i \in \rr$ for $S_{IJ,(t_1, \ldots, t_k)}$. 

Varying these $t_i \in \rr$ gives a $k$ dimensional family of cobordisms in $S^3 \times \rr$. Restricting to $\sum t_i = 0$ gives a $k-1$ dimensional family.

Note that the double branched cover of $B_i$ at the corresponding arcs in $L_I$ is an $S^1 \times D^2$, and the double branched cover of its boundary $S^2$ at the four points of its intersection with $L_I$ is $S^1 \times S^1$. The cobordism $W(IJ)$ can be seen as $Y(I) \times (-\infty, -\tau] \cup Y(J) \times [\tau, \infty)$ with $k$ 2-handles attached between times $-\tau$ and $\tau$, with these two handles attached at $S^1 \times D^2$s that are disjoint from each other.

As in the above description of embeddings of $S_{IJ,(t_1,\ldots, t_k)}$ in $S^3 \times \rr$ parametrized by tuples $(t_1,\ldots, t_k)$ representing the times at which the saddles are added, we may consider families of metrics on $W(IJ)$ parametrized by tuples $(t_1,\ldots, t_k) \in \rr^k$, where $t_i$ represents the $\rr$ coordinate of when the $i$th $2$-handle is attached. 

Varying these $t_i \in \rr$ gives a $k$ dimensional family of metrics on $W(IJ)$. Restricting to $\sum t_i = 0$ gives a $k-1$ dimensional family. Bloom's spectral sequence is built by counting monopoles on families of metrics on $W(IJ)$ of this form. That is, they are parametrized over a $k-1$ dimensional family $\{(t_1, \ldots, t_k) \in \rr^k| \sum t_i = 0\}$, where the $t_i$ correspond to the time at which the $i$th $2$-handle is attached. That is, the family of metrics parametrizes translation of handles in the $\rr$ direction.

It was shown in Proposition 4.6.1 of the book, \cite{KM_monopole_book}, that on a closed manifold with positive scalar curvature, all solutions to the monopole equations are reducible. In the context of the cobordisms in the aforementioned spectral sequence, this means that if we could choose the family of metrics parametrising translation of the handles in the $\rr$ direction such that each metric in the family has positive scalar curvature, then the maps $\wtl{D}^I_J$ contain only terms coming from reducible solutions, which should make them easier to calculate.

\section{What the cobordisms look like}\label{cobordism_form}


Let us consider the cobordism $W(IJ)$ in the notation of the previous section, a cobordism from $Y(I)$ to $Y(J)$, where $Y(I)$ and $Y(J)$ are the double-branched covers of the resolutions $L_I$ and $L_J$ in $S^3$. 

In addition to $L_I$ and $L_J$, we will consider the resolution $L_0$, defined in the introduction, which is the resolution of $L_0$ such that the boundary of each black region is one of the unlink components.


We can consider $L_0$ as a collection of circles $c_0,\ldots c_n$ in $S^2$, such that the circle corresponding to $c_0$ has all the other circles inside it, and for $i,j>0$, $c_i$ and $c_j$ do not contain each other. For example, for the unlink projection from Figure \ref{checkerboard}, the circles are depicted in Figure \ref{loops_with_crossings2}.

\begin{figure}[ht!]
\centering
\includegraphics[width=80mm]{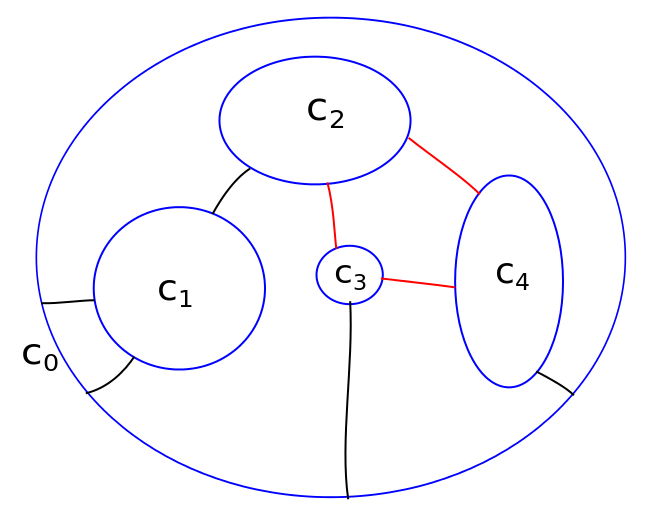}
\caption[A resolution of a link with arcs representing crossings changed in cobordism]{The $L_0$ for the $L8N1$ projection, drawn with $c_0$ containing the others.}
\label{loops_with_crossings2}
\end{figure}

The double branched cover $Y_0$ of $(S^3,L_0)$ can be thought of as $\#^n(S^1 \times S^2)$, where for $i\in \{1,\ldots n\}$, $c_i$ sits as $\{0\} \times \gamma \subset S^1 \times S^2$ in the $i$th component, where $\gamma \subset S^2$ is an equator, and $c_0$ sits as $\#^n \{\pi\} \times \gamma$. In this picture, the crossings between the $c_i$ and $c_0$ for $i \neq 0$ can be thought of as certain points on the copies of $\{0\} \times \gamma$, and their corresponding points in the $\{\pi\} \times \gamma$. Each of these crossings is over a point $x_i \in S^2$, where we have projected to the $S^2$ component of $S^1 \times S^2$.

For example, if $n=1$, ie there are only two circles, then the double branched cover is $S^1 \times S^2$, as depicted in Figure \ref{S1xS2}. In the figure, as we go around the depicted great circle in $S^2$, the magenta star traces out $c_1$ and the blue star traces out $c_0$. Crossings correspond to specific points $x_l$ on the equator of the $S^2$, and the embedding $\phi:S^1 \times D^2 \to Y$ of the corresponding handle attachment has image given by the $S^1$ bundle over a disk $D_l$ around $x_l$.

\begin{figure}[ht!]
\centering
\includegraphics[width=60mm]{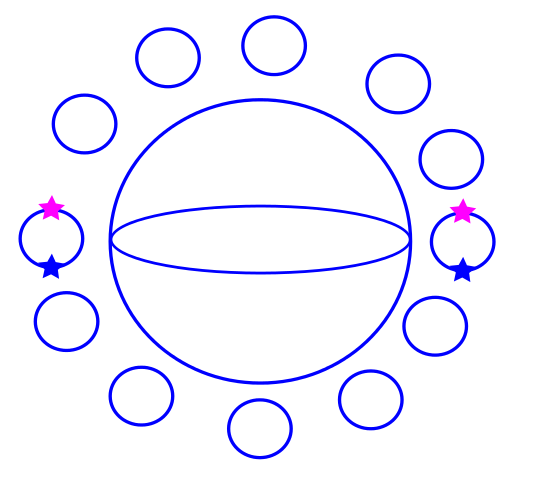}
\caption[$S^1 \times S^2$, the double branched cover of the two-component unlink]{The double branched cover for $c_0 \cup c_1$. }
\label{S1xS2}
\end{figure}

For a cobordism that satisfies the conditions of the theorem, the $L_0$ can be drawn as a circle, $c_0$, with smaller circles $c_1,\ldots c_{l}$ inside it, where all crossings are between $c_0$ and $c_i$ for $i>0$, as in Figure \ref{okay_graph}. If the circles are $c_0, \ldots c_l$ with all edges having one end on $c_0$, then the double branched cover is the connected sum of $l$ copies of Figure \ref{okay_graph}. We depict these crossings as arcs between circles, with black arcs depicting where $L_0$ and $L_I$ are resolved the same way and red arcs depicting where they are resolved differently. In particular, $L_I$ looks like $L_0$ but with a band added along each red arc, and $L_J$ looks like $L_0$ but with a band added along each black arc. Thus, in the case of the $L_0$ of Figure \ref{okay_graph}, the resolution $L_I$ is depicted in Figure \ref{okay_graph_LI}.

\begin{figure}[ht!]
\centering
\includegraphics[width=80mm]{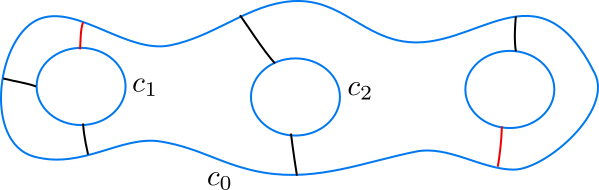}
\caption{The circles $c_0, c_1, \ldots , c_n$ of $L_0$, with an edge shown for each crossing at which $L_I$ and $L_J$ differ. The arcs drawn between the circles correspond to crossings, with black arcs depicting where $L_0$ and $L_I$ are resolved the same way and red arcs depicting where they are resolved differently.}
\label{okay_graph}
\end{figure}

\begin{figure}[ht!]
\centering
\includegraphics[width=80mm]{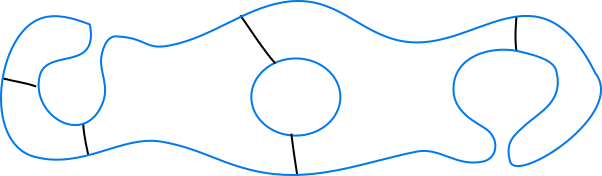}
\caption{The $L_I$ resolution for the $L_0$ resolution depicted in Figure \ref{okay_graph}.}
\label{okay_graph_LI}
\end{figure}

Stretching out the parts between the $c_i$, we can write the double branched cover of this as the connected sum $(S^1 \times S^2)^{\# l}$, where the $i$th copy of $S^1 \times S^2$ corresponds to the double branched cover of $c_0 \coprod c_i$, and the neck between copies of $S^1 \times S^2$, which is an $S^2 \times I$, is the double branched cover of $S^2 \times I$ with respect to $(\text{2 points}) \times I$, where the two points are where an $S^2$ meets $c_0$, as in Figure \ref{okay_graph_cut}.

\begin{figure}[ht!]
\centering
\includegraphics[width=80mm]{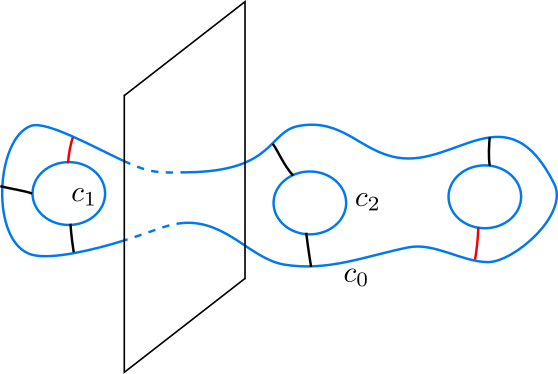}
\caption{The double branched cover of the $S^2$ (depicted as a plane) along the two points at which it meets $c_0$ is an $S^2$.}
\label{okay_graph_cut}
\end{figure}

Let $L_1$ be the resolution of $L$ with all crossings resolved in the opposite way from $L_0$. Then the cobordism from $L_0$ to $L_1$ consists of adding a 1-handle along where each of the crossings are marked. In the double branched cover, each of the arcs correspond to a circle, and the attachment of the 1-handle along the arc corresponds to attachment of a 2-handle along an $S^1 \times D^2$ tubular neighbourhood of the circle.

Then the double branched cover of $L_I$ is that of $L_0$ with $2$-handles attached along the circle pre-images of the red arcs and the double branched cover of $L_J$ is that of $L_0$ with $2$-handles attached along the circle pre-images of the black arcs.

The cobordism $W(IJ)$ can be seen as reversing the red-arc handle attachments from $L_I$ to $L_0$, and then attaching the black-arc handles from $L_0$ to $L_J$. The double branched cover for the $L_0$ of Figure \ref{okay_graph} is depicted in Figure \ref{S1xS2_connect_sums3}, along with the circles along whose tubular neighbourhoods 2-handles are attached.

\begin{figure}[ht!]
\centering
\includegraphics[width=100mm]{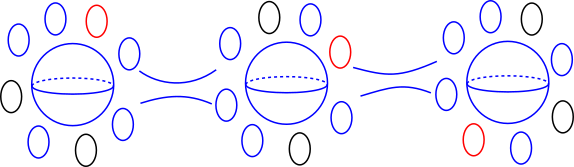}
\caption{The double branched cover of the $L_0$ in \ref{okay_graph}. The double branched cover of $L_I$ is obtained from this by attaching a $2$-handle along a tubular neighbourhood of each red circle, and the double branched cover of $L_J$ is obtained from this by attaching a $2$-handle along a tubular neighbourhood of each black circle.}
\label{S1xS2_connect_sums3}
\end{figure}

\section{Metrics of positive scalar curvature}

Let $Y(0)$ and $Y(1)$ be the double branched covers of $L_0$ and $L_1$. As discussed in the previous section, $Y(I)$ looks like $Y(0)$ of $L_0$, which is a connected sum of $S^1 \times S^2$s, but with 2-handles attached at the tubular neighbourhoods of certain fibre $S^1$s, and $Y(J)$ is similar, with the two-handles attached at other fibre $S^1$s.

The product metric on $S^1 \times S^2$ where the metric on $S^2$ is the standard sphere metric has positive scalar curvature.

It was shown by Gromov and Lawson in \cite{GL_psc} and by Schoen and Yau in \cite{SY_psc} that for manifolds $M$ and $M'$ if $M'$ is obtained from $M$ by surgery in codimension at least $3$ and $M$ has a metric of positive scalar curvature, then $M'$ does as well. In the case of the connected sum construction for 3-manifolds $M_1$ and $M_2$, which we can see as $0$ surgery on the 3-manifold $M_1 \coprod M_2$, the new metric constructed on $M_1 \# M_2$ can be taken to be the same as the original metric on $M_1$ and $M_2$ away from a small $B^3$ in each; the connected sum operation is being performed in the balls.

Thus there is a metric of positive scalar curvature on $\#^n S^1 \times S^2$ which is the product metric on each $S^1 \times S^2$s away from small balls at which it is attached to the other $S^1 \times S^2$s. 

We will construct metrics on the $2$-handles $D^2 \times D^2$ attached along $S^1 \times D^2$ tubular neighbourhoods of fibres $S^1$ that agree with this metric on $S^1 \times D^2 \subset S^1 \times S^2$. Moreover we will show that it is possible to construct such metrics for arbitrarily small disks $D^2 \subset S^2$, where ``arbitrarily small'' here means we will take $D^2 \subset S^2$ given by
\[D^2= \{\(\cos (\theta ) \cos (\phi ),\cos (\theta ) \sin (\phi ),\sin (\theta )\)| \theta>\theta_0\} \subset S^2\]
for $\theta_0$ arbitrarily close to $\pi/2$, that is for caps centered at the north pole of the $S^2$ with arbitrarily small radius relative to the radius of the sphere.

This gives a metric on the cobordism from $Y(0)$ to $Y(1)$ given by attaching $2$ handle to each of the circles corresponding to crossings. Moreover, since the metrics are constant outside of the small tubular neighbourhoods of the attaching circles, it gives a family of metrics from $Y(0)$ to $Y(1)$ parametrized by the time at which the handles are attached. It also gives a family of metrics on the cobordism $W(IJ)$ parametrized by the same, because the parts inside the boundary $S^1 \times S^1$ of the tubular neighbourhoods of the attaching $S^1$s can be run backwards from the attached version (the $1$-resolution) to the un-attached version (the $0$-resolution).

\subsection{Surgery on the metric $S^1 \times D^2 \subset S^1 \times S^2$}\label{basic_handle}

As mentioned above, for arbitrarily small $D^2 \subset S^2$, we construct a metric of positive scalar curvature on the attaching handle on $S^1 \times D^2 \subset S^1 \times S^2$ that agrees with the standard metric (the product metric with the standard sphere metric on the $D^2$) near the boundary.

\begin{figure}[ht!]
\centering
\includegraphics[width=80mm]{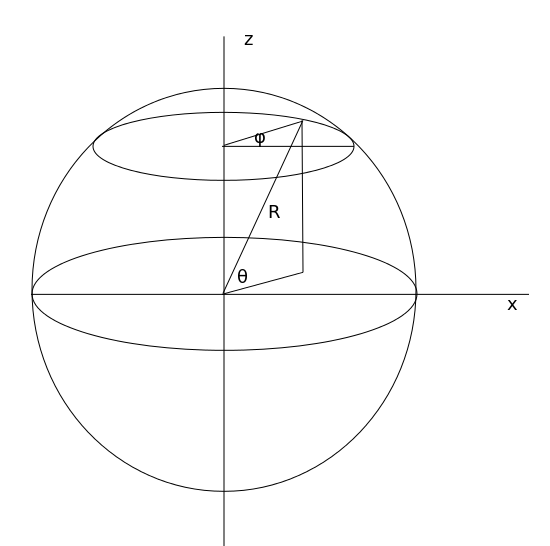}
\caption[]{}
\label{polar_coords}
\end{figure}

Let $\theta_0$ and $R$ be positive real numbers with $\theta_0<\pi/2$.

In a sphere $S^2_R$ of radius $R$, let $D^2_{R,\theta_0}$ be the cap in $S^2_R$ spanning angle $\pi-2\theta_0$, that is, in spherical coordinates, $D^2_{R,\theta_0}$ is given by $\{(\theta,\phi)|\theta>\theta_0, \phi \in [0, 2 \pi]\}$. In this section, we give a metric on the $1$ surgery associated to the $S^1$ over the north pole, viewed as a $D^2 \times D^2$ attached to $S^1 \times D^2_{R, \theta_0} \subset S^1 \times S^2_R$. 

Parametrising the cobordism using a ``time'' axis $t$, the metric looks like this: for time $t<-T$, it is cylindrical metric $S^1 \times D^2 \times \rr$, in the middle region where the cobordism is happening, $-T<t<T$, it is the submanifold metric on the submanifold of $\rr^2 \times D^2 \times \rr$, where we take a circle in the $\rr^2$ at each time, but as time increases, the radius of the circle over the north pole in $D^2$ decreases, so that for $t>T$, we will have a cylindrical metric like that on the upper hemisphere in Figure \ref{basic_handle}; the latter will be the metric on the other end of the cobordism. 

\begin{figure}[ht!]
\centering
\includegraphics[width=80mm]{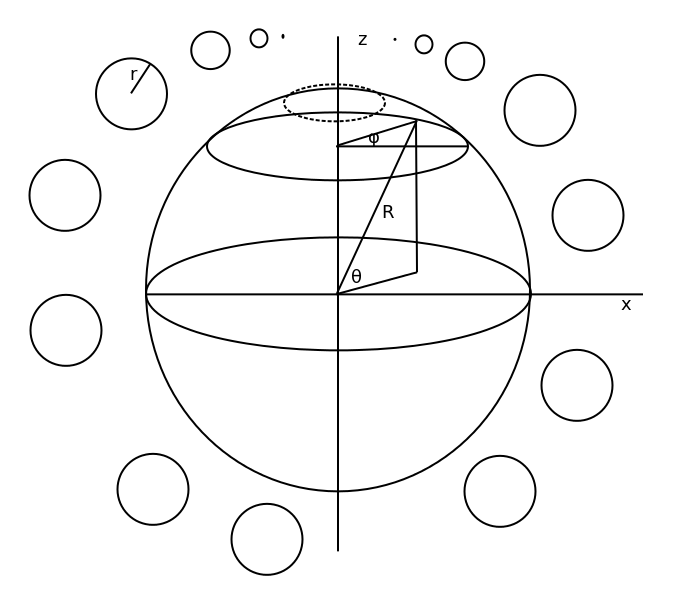}
\caption{An $S^3$ obtained by attaching a handle to $S^1 \times S^2$. The cobordism between $S^1 \times S^2$ and this $S^3$, if seen as a movie, involves the $S^1$s over the north pole becoming progressively smaller until they cease to exist.}
\label{basic_handle}
\end{figure}

Let $r:[-\pi/2,\pi/2]  \times \rr \to \rr$ be a smooth function such that there are no $(\theta,t)$ for which $r$, ${r_\theta}$, and ${r_t}$ all vanish, and such that $r$ is constant in $\theta$ for an open neighbourhood of $\theta = \pi/2$.

Consider $X \subset \rr^6$ cut out by the equations $h(x,y,z,u,w,t)=(0,0)$, for $h:\rr^6 \to \rr^2$ given by
\[h(x,y,z,u,w,t) = (x^2+y^2+z^2-R^2,u^2+w^2-r(\theta(x,y,z),t))\]
where $\theta(x,y,z) =\arctan(z/\sqrt{x^2+y^2})$.

The manifold cut out is the cobordism described above: the $x^2+y^2+z^2=R^2$ part cuts out the $S^2$, and the $u^2+w^2-r(\theta(x,y,z),t)$ part signifies how the radii of the circles over each point in $S^2$ change as you move along $S^2$ (closer, or farther away from the north pole) and along time.

\begin{condition}\label{r_boundary_conditions} Fix $R$ and $\theta_0$. In terms of $r$, the boundary conditions on the cobordism are that it:
\begin{enumerate}
\item is constant in $t$ for $t < -T$ and for $t > T$
\item is constant in both $t$ and $\theta$ for $\theta \leq \theta_0$ or $t \leq -T$.
\item is constant in $\theta$ for $\theta \geq \pi/2-\varepsilon_1$ for some positive $\varepsilon_1$.
\item For $t>T$,  $r(\pi/2,t)<0$
\item Is non-increasing in $\theta$ and $t$.
\item There are no values of $(\theta,t)$ for which $r_t, r_\theta, r$ all vanish.
\end{enumerate}
\end{condition}

We start by showing that the equations cut out a smooth submanifold of $\rr^6$. 

Note that
\[dh(x,y,z,u,w,t) = {\left[
\begin{array}{cccccc} 
2x  & 2y & 2z & 0 & 0 & 0 \\  
{-r_\theta} \ddx{\theta}{x} & {-r_\theta} \ddx{\theta}{y} &  {-r_\theta} \ddx{\theta}{z} & 2u & 2w & {-r_t} \\ 
\end{array} 
\right]}\]
\[=\left[\begin{array}{cccccc} 
2x  & 2y & 2z & 0 & 0 & 0 \\  
\frac{r_\theta}{R^2}\cdot\frac{x z}{\sqrt{x^2+y^2}} & \frac{r_\theta}{R^2}\cdot\frac{y z}{\sqrt{x^2+y^2}} &  -\frac{r_\theta}{R^2} \sqrt{x^2+y^2} & 2u & 2w & {-r_t} \\ 
\end{array}\right]\]
and this derivative is smooth in an open neighbourhood of $\theta=\pi/2$ because $\(\ddx{}{\theta}\)^ir=0$ for $i \geq 1$, and $r_t$ has no dependence on $x,y,z,u,w$ and it is smooth in $t$. For $\theta$ away from $\frac{\pi}{2}$, it is easy to see that $dh$ is smooth on $X$, because $r$, $\sqrt{x^2 + y^2}$, $\frac{xz}{\sqrt{x^2+y^2}}$, and $\frac{xz}{\sqrt{x^2+y^2}}$ are all smooth away from an open neighbourhood of $\theta = \pi/2$. 


Let us show that $dh$ has rank $2$. If not, then we must have $r_t = 0$ and $u=w=0$, so $r=0$. Thus, by our assumption that $r,r_\theta,$ and $r_t$ do not simultaneously vanish, we have $r_\theta \neq 0$. However,
\[\frac{r_\theta}{R^2}\cdot\frac{x z}{\sqrt{x^2+y^2}} , \frac{r_\theta}{R^2}\cdot\frac{y z}{\sqrt{x^2+y^2}} ,  -\frac{r_\theta}{R^2} \sqrt{x^2+y^2}\]
is a vector of length $\frac{r_\theta}{R}$ orthogonal to $(2x,2y,2z)$, so it cannot be parallel to it.

Thus, $dh$ has rank $2$, so the submanifold cut out by $h(x,y,z,u,w,t)=0$ is smooth.

We now construct a smooth function $r$, satisfying the above boundary conditions such that the cobordism cut out by $h$ has positive scalar curvature. First, we construct such a function when we are allowed to vary $R$: 

\begin{lemma} For any positive $\theta_0<\pi/2$, for sufficiently large $R>0$ there is a function $r:[0,\pi/2] \times \rr \to \rr$ satisfying the boundary conditions above, such that the cobordism cut out has postive scalar curvature.
\end{lemma}

\begin{proof}

Away from the north and south poles and away from $r=0$, parametrize the submanifold by:
\[(\theta,\phi,t,\alpha) \mapsto \(R \cos (\theta ) \cos (\phi ),R \cos (\theta ) \sin (\phi ),R \sin (\theta ),\sqrt{r} \cos (\alpha),\sqrt{r} \sin (\alpha),t\)\]

Recall that for an embedding $X$ of a submanifold we have that the induced metric tensor defined on the submanifold is given by
\[g_{ab} = \partial_a X^\mu \partial_b X^\nu g_{\mu\nu}.\]

Thus, away from $\theta = \pm \frac{\pi}{2}$ and away from $r=0$ metric $g_{\mu\nu}$ in $(\theta, \phi, t, \alpha)$ coordinates is 
\[
g_{\mu\nu}=\left(
\begin{array}{cccc}
 R^2+\frac{1}{4 r}{r_\theta}^2 & 0 & \frac{1}{4 r} {r_t} {r_\theta}& 0 \\
 0 & R^2 \cos ^2(\theta ) & 0 & 0 \\
 \frac{1}{4 r} {r_t} {r_\theta}& 0 & \frac{1}{4 r}{r_t}^2+1 & 0 \\
 0 & 0 & 0 & r \\
\end{array}
\right)
\]
and $g^{\mu \nu}$ is
\[
g^{\mu \nu}=\left(
\begin{array}{cccc}
 \frac{{r_t}^2+4 r}{{r_t}^2 R^2+4 r R^2+{r_\theta}^2} & 0 & -\frac{{r_t} {r_\theta}}{{r_t}^2 R^2+4 r R^2+{r_\theta}^2} & 0 \\
 0 & \frac{\sec ^2(\theta )}{R^2} & 0 & 0 \\
 -\frac{{r_t} {r_\theta}}{{r_t}^2 R^2+4 r R^2+{r_\theta}^2} & 0 & \frac{4 r R^2+{r_\theta}^2}{{r_t}^2 R^2+4 r R^2+{r_\theta}^2} & 0 \\
 0 & 0 & 0 & \frac{1}{r} \\
\end{array}
\right)
\]

Next we compute the $\Gamma^{k}_{ij}$ using formula
\[ \Gamma^{k}_{ij} = \sum_m \frac{1}{2}g^{km}\(\partial_jg_{im}+\partial_ig_{jm}-\partial_mg_{ij}\).\] These are given by the following matrices, with indices in order $\theta,\phi,t,\alpha$. 

\[\Gamma^{\theta}_{ij} = \left(
\begin{array}{cccc}
 -\frac{{r_\theta} \left({r_\theta}^2-2 r {r_{\theta\theta}}\right)}{2 r \left({r_t}^2 R^2+4 r R^2+{r_\theta}^2\right)} & 0 & \frac{{r_\theta} (2 r {r_{\theta t}}-{r_t} {r_\theta})}{2 r
   \left({r_t}^2 R^2+4 r R^2+{r_\theta}^2\right)} & 0 \\
 0 & \frac{R^2 \left({r_t}^2+4 r\right) \cos (\theta) \sin (\theta)}{{r_t}^2 R^2+4 r R^2+{r_\theta}^2} & 0 & 0 \\
 \frac{{r_\theta} (2 r {r_{\theta t}}-{r_t} {r_\theta})}{2 r \left({r_t}^2 R^2+4 r R^2+{r_\theta}^2\right)} & 0 & -\frac{\left({r_t}^2-2 r {r_{tt}}\right) {r_\theta}}{2 r
   \left({r_t}^2 R^2+4 r R^2+{r_\theta}^2\right)} & 0 \\
 0 & 0 & 0 & -\frac{2 r {r_\theta}}{{r_t}^2 R^2+4 r R^2+{r_\theta}^2} \\
\end{array}
\right)\]

\[\Gamma^{\phi}_{ij}= \left(
\begin{array}{cccc}
 0 & -\tan (\theta) & 0 & 0 \\
 -\tan (\theta) & 0 & 0 & 0 \\
 0 & 0 & 0 & 0 \\
 0 & 0 & 0 & 0 \\
\end{array}
\right)\]

\[\Gamma^{t}_{ij} = \left(
\begin{array}{cccc}
 -\frac{R^2 {r_t} \left({r_\theta}^2-2 r {r_{\theta\theta}}\right)}{2 r \left({r_t}^2 R^2+4 r R^2+{r_\theta}^2\right)} & 0 & \frac{R^2 {r_t} (2 r {r_{\theta t}}-{r_t} {r_\theta})}{2 r
   \left({r_t}^2 R^2+4 r R^2+{r_\theta}^2\right)} & 0 \\
 0 & -\frac{R^2 {r_t} {r_\theta} \cos (\theta) \sin (\theta)}{{r_t}^2 R^2+4 r R^2+{r_\theta}^2} & 0 & 0 \\
 \frac{R^2 {r_t} (2 r {r_{\theta t}}-{r_t} {r_\theta})}{2 r \left({r_t}^2 R^2+4 r R^2+{r_\theta}^2\right)} & 0 & -\frac{R^2 {r_t} \left({r_t}^2-2 r {r_{tt}}\right)}{2 r
   \left({r_t}^2 R^2+4 r R^2+{r_\theta}^2\right)} & 0 \\
 0 & 0 & 0 & -\frac{2 r R^2 {r_t}}{{r_t}^2 R^2+4 r R^2+{r_\theta}^2} \\
\end{array}
\right)\]

\[\Gamma^{\alpha}_{ij}=\left(
\begin{array}{cccc}
 0 & 0 & 0 & \frac{{r_\theta}}{2 r} \\
 0 & 0 & 0 & 0 \\
 0 & 0 & 0 & \frac{{r_t}}{2 r} \\
 \frac{{r_\theta}}{2 r} & 0 & \frac{{r_t}}{2 r} & 0 \\
\end{array}
\right) \]

Next, we compute the Riemann curvature tensor, using the formula
\[R^\rho_{\sigma \mu \nu}  = \partial_\mu \Gamma^\rho_{\nu \sigma} - \partial_\nu \Gamma^\rho_{\mu\sigma} + \Gamma^\rho_{\mu\lambda}\Gamma^{\lambda}_{\nu\sigma}-\Gamma^\rho_{\nu\lambda}\Gamma^{\lambda}_{\mu\sigma}.\]

These are:
\[R^\theta_{\theta \theta \theta}=R^\theta_{\phi \theta \theta}=R^\theta_{t \theta \theta}=R^\theta_{\alpha \theta \theta}=R^\theta_{\theta \theta \phi}=0\]

\[R^\theta_{\phi \theta \phi}=\frac{R^2 \cos (\theta)}{\left(R^2 r_{t}^2+r_{\theta}^2+4 R^2 {r}\right)^2}\]
\[\times\left(R^2 \cos (\theta) r_{t}^4+4 {r} \left(2 R^2 \cos (\theta) r_{t}^2+r_{\theta} \left(\cos (\theta) r_{\theta}-\sin (\theta) r_{\theta\theta}\right)\right)\right.\]
\[+\left.\sin (\theta) r_{\theta}^2 r_{\theta t} r_{t}+2 \sin (\theta) r_{\theta}^3+r_{\theta} r_{t}^2 \left(\cos (\theta) r_{\theta}-\sin (\theta) r_{\theta\theta}\right)+16 R^2 \cos (\theta) {r}^2\right)\]

\[R^\theta_{t \theta \phi}=R^\theta_{\alpha \theta \phi}=0\]

\[R^\theta_{\theta \theta t}=-\frac{R^2 r_{t} r_{\theta} \left(r_{\theta\theta} r_{t}^2-2 r_{\theta} r_{\theta t} r_{t}+2 {r} r_{\theta t}^2+r_{t t} \left(r_{\theta}^2-2 {r} r_{\theta\theta}\right)\right)}{2 {r} \left(R^2 r_{t}^2+r_{\theta}^2+4 R^2 {r}\right)^2}\]

\[R^\theta_{\phi \theta t}=0\]

\[R^\theta_{t \theta t}=-\frac{R^2 \left(r_{t}^2+4 {r}\right) \left(r_{\theta\theta} r_{t}^2-2 r_{\theta} r_{\theta t} r_{t}+2 {r} r_{\theta t}^2+r_{t t} \left(r_{\theta}^2-2 {r} r_{\theta\theta}\right)\right)}{2 {r} \left(R^2 r_{t}^2+r_{\theta}^2+4 R^2 {r}\right)^2}\]

\[R^\theta_{\alpha \theta t}=[R^\theta_{\theta \theta \alpha}=R^\theta_{\phi \theta \alpha}= R^\theta_{t \theta \alpha}=0\]

\[R^\theta_{\alpha \theta \alpha}=\frac{2 R^2 {r} \left(2 r_{\theta}^2+r_{t} r_{\theta t} r_{\theta}-\left(r_{t}^2+4 {r}\right) r_{\theta\theta}\right)}{\left(R^2 r_{t}^2+r_{\theta}^2+4 R^2 {r}\right)^2}\]

\[R^\phi_{\theta \phi \theta}=\frac{\tan (\theta) \left(r_{\theta}^3-2 {r} r_{\theta} r_{\theta\theta}\right)}{2 {r} \left(R^2 r_{t}^2+r_{\theta}^2+4 R^2 {r}\right)}-\tan ^2(\theta)+\sec ^2(\theta)\]

\[R^\phi_{\phi \phi \theta}=0\]

\[R^\phi_{t \phi \theta}=\frac{\tan (\theta) r_{\theta} \left(r_{t} r_{\theta}-2 {r} r_{\theta t}\right)}{2 {r} \left(R^2 r_{t}^2+r_{\theta}^2+4 R^2 {r}\right)}\]

\[R^\phi_{\alpha \phi \theta}=R^\phi_{\theta \phi \phi}=R^\phi_{\phi \phi \phi}=R^\phi_{t \phi \phi}=R^\phi_{\alpha \phi \phi}=0\]

\[R^\phi_{\theta \phi t}=\frac{\tan (\theta) r_{\theta} \left(r_{t} r_{\theta}-2 {r} r_{\theta t}\right)}{2 {r} \left(R^2 r_{t}^2+r_{\theta}^2+4 R^2 {r}\right)}\]

\[R^\phi_{\phi \phi t}=0\]

\[R^\phi_{t \phi t}=\frac{\tan (\theta) \left(r_{t}^2-2 {r} r_{t t}\right) r_{\theta}}{2 {r} \left(R^2 r_{t}^2+r_{\theta}^2+4 R^2 {r}\right)}\]

\[R^\phi_{\alpha \phi t}=R^\phi_{\theta \phi \alpha}=R^\phi_{\phi \phi \alpha}=R^\phi_{t \phi \alpha}=0\]

\[R^\phi_{\alpha \phi \alpha}=\frac{2 \tan (\theta) {r} r_{\theta}}{R^2 r_{t}^2+r_{\theta}^2+4 R^2 {r}}\]

\[R^t_{\theta t \theta}=-\frac{R^2 \left(r_{\theta\theta} r_{t}^2-2 r_{\theta} r_{\theta t} r_{t}+2 {r} r_{\theta t}^2+r_{t t} \left(r_{\theta}^2-2 {r} r_{\theta\theta}\right)\right) \left(r_{\theta}^2+4 R^2 {r}\right)}{2 {r} \left(R^2 r_{t}^2+r_{\theta}^2+4 R^2 {r}\right)^2}\]

\[R^t_{\phi t \theta}=0\]

\[R^t_{t t \theta}=-\frac{R^2 r_{t} r_{\theta} \left(r_{\theta\theta} r_{t}^2-2 r_{\theta} r_{\theta t} r_{t}+2 {r} r_{\theta t}^2+r_{t t} \left(r_{\theta}^2-2 {r} r_{\theta\theta}\right)\right)}{2 {r} \left(R^2 r_{t}^2+r_{\theta}^2+4 R^2 {r}\right)^2}\]

\[R^t_{\alpha t \theta}=R^t_{\theta t \phi}=0\]

\[R^t_{\phi t \phi}=\frac{R^2 \sin (2 \theta) r_{\theta} \left(2 R^2 r_{t}^2-r_{t t} \left(r_{\theta}^2+4 R^2 {r}\right)+r_{\theta} r_{\theta t} r_{t}\right)}{2 \left(R^2 r_{t}^2+r_{\theta}^2+4 R^2 {r}\right)^2}\]

\[R^t_{t t \phi}=R^t_{\alpha t \phi}=R^t_{\theta t t}=R^t_{\phi t t}=R^t_{t t t}=R^t_{\alpha t t}=R^t_{\theta t \alpha}=R^t_{\phi t \alpha}=R^t_{t t \alpha}=0\]

\[R^t_{\alpha t \alpha}=\frac{2 R^2 {r} \left(2 R^2 r_{t}^2-r_{t t} \left(r_{\theta}^2+4 R^2 {r}\right)+r_{\theta} r_{\theta t} r_{t}\right)}{\left(R^2 r_{t}^2+r_{\theta}^2+4 R^2 {r}\right)^2}\]

\[R^\alpha_{\theta \alpha \theta}=\frac{R^2 \left(r_{\theta}^2-2 {r} r_{\theta\theta}\right)}{{r} \left(R^2 r_{t}^2+r_{\theta}^2+4 R^2 {r}\right)}\]

\[R^\alpha_{\phi \alpha \theta}=0\]

\[R^\alpha_{t \alpha \theta}=\frac{R^2 \left(r_{t} r_{\theta}-2 {r} r_{\theta t}\right)}{{r} \left(R^2 r_{t}^2+r_{\theta}^2+4 R^2 {r}\right)}\]

\[R^\alpha_{\alpha \alpha \theta}=R^\alpha_{\theta \alpha \phi}=0\]

\[R^\alpha_{\phi \alpha \phi}=\frac{R^2 \sin (2 \theta) r_{\theta}}{R^2 r_{t}^2+r_{\theta}^2+4 R^2 {r}}\]

\[R^\alpha_{t \alpha \phi}=R^\alpha_{\alpha \alpha \phi}=0\]

\[R^\alpha_{\theta \alpha t}=\frac{R^2 \left(r_{t} r_{\theta}-2 {r} r_{\theta t}\right)}{{r} \left(R^2 r_{t}^2+r_{\theta}^2+4 R^2 {r}\right)}\]

\[R^\alpha_{\phi \alpha t}=0\]

\[R^\alpha_{t \alpha t}=\frac{R^2 \left(r_{t}^2-2 {r} r_{t t}\right)}{{r} \left(R^2 r_{t}^2+r_{\theta}^2+4 R^2 {r}\right)}\]

\[R^\alpha_{\alpha \alpha t}=R^\alpha_{\theta \alpha \alpha}=R^\alpha_{\phi \alpha \alpha}=R^\alpha_{t \alpha \alpha}=R^\alpha_{\alpha \alpha \alpha}=0.\]

The Ricci curvature, with formula $R_{ij} = \sum_{k} R^{k}_{ikj}$, is given by

\[\text{Ric}_{\theta \theta}=\frac{R^2 \left(r_{\theta}^2-2 {r} r_{\theta \theta}\right)}{{r} \left(R^2 r_{t}^2+r_{\theta}^2+4 R^2 {r}\right)}-\frac{R^2 \left(r_{\theta \theta} r_{t}^2-2 r_{\theta} r_{\theta t} r_{t}+2 {r} r_{\theta t}^2+r_{t t} \left(r_{\theta}^2-2 {r} r_{\theta \theta}\right)\right) \left(r_{\theta}^2+4 R^2 {r}\right)}{2 {r} \left(R^2 r_{t}^2+r_{\theta}^2+4 R^2 {r}\right)^2}\]
\[+\frac{\tan (\theta) \left(r_{\theta}^3-2 {r} r_{\theta} r_{\theta \theta}\right)}{2 {r} \left(R^2 r_{t}^2+r_{\theta}^2+4 R^2 {r}\right)}-\tan ^2(\theta)+\sec ^2(\theta)\]

\[\text{Ric}_{\phi \theta}=0\]

\[\text{Ric}_{t \theta}=\frac{1}{2 {r} \left(R^2 r_{t}^2+r_{\theta}^2+4 R^2 {r}\right)^2} \times\]
\[ 
\Big(2 R^2 \left(r_{t} r_{\theta}-2 {r} r_{\theta t}\right) \left(R^2 r_{t}^2+r_{\theta}^2+4 R^2 {r}\right)-R^2 r_{t} r_{\theta} \left(r_{\theta \theta} r_{t}^2-2 r_{\theta} r_{\theta t} r_{t}+2 {r} r_{\theta t}^2+r_{t t} \left(r_{\theta}^2-2 {r} r_{\theta \theta}\right)\right)\]
\[+\tan (\theta) r_{\theta} \left(r_{t} r_{\theta}-2 {r} r_{\theta t}\right) \left(R^2 r_{t}^2+r_{\theta}^2+4 R^2 {r}\right) \Big)\]

\[\text{Ric}_{\alpha \theta}=0\]

\[\text{Ric}_{\theta \phi}=0\]

\[\text{Ric}_{\phi \phi}=\frac{R^2 \cos (\theta)}{\left(R^2 r_{t}^2+r_{\theta}^2+4 R^2 {r}\right)^2} \times\]
\[\Big(R^2 \cos (\theta) r_{t}^4+r_{\theta} r_{t}^2 \left(-\sin (\theta) r_{\theta \theta}+\cos (\theta) r_{\theta}+4 R^2 \sin (\theta)\right)\]
\[+4 {r} \left(2 R^2 \cos (\theta) r_{t}^2+r_{\theta} \left(-R^2 \sin (\theta) r_{t t}-\sin (\theta) r_{\theta \theta}+\cos (\theta) r_{\theta}+2 R^2 \sin (\theta)\right)\right)\]
\[+2 \sin (\theta) r_{\theta}^2 r_{\theta t} r_{t}-\sin (\theta) \left(r_{t t}-4\right) r_{\theta}^3+16 R^2 \cos (\theta) {r}^2\Big)\]

\[\text{Ric}_{t \phi}=0\]

\[\text{Ric}_{\alpha \phi}=0\]

\[\text{Ric}_{\theta t}=\frac{1}{{2 {r} \left(R^2 r_{t}^2+r_{\theta}^2+4 R^2 {r}\right)^2}} \times \]
\[\Big(2 R^2 \left(r_{t} r_{\theta}-2 {r} r_{\theta t}\right) \left(R^2 r_{t}^2+r_{\theta}^2+4 R^2 {r}\right)-R^2 r_{t} r_{\theta} \left(r_{\theta \theta} r_{t}^2-2 r_{\theta} r_{\theta t} r_{t}+2 {r} r_{\theta t}^2+r_{t t} \left(r_{\theta}^2-2 {r} r_{\theta \theta}\right)\right)\]
\[+\tan (\theta) r_{\theta} \left(r_{t} r_{\theta}-2 {r} r_{\theta t}\right) \left(R^2 r_{t}^2+r_{\theta}^2+4 R^2 {r}\right)\Big)\]

\[\text{Ric}_{\phi t}=0\]

\[\text{Ric}_{t t}=\frac{1}{2 {r} \left(R^2 r_{t}^2+r_{\theta}^2+4 R^2 {r}\right)^2}
\times
\]
\[\Big(2 R^2 \left(r_{t}^2-2 {r} r_{t t}\right) \left(R^2 r_{t}^2+r_{\theta}^2+4 R^2 {r}\right)-R^2 \left(r_{t}^2+4 {r}\right) \left(r_{\theta \theta} r_{t}^2-2 r_{\theta} r_{\theta t} r_{t}+2 {r} r_{\theta t}^2+r_{t t} \left(r_{\theta}^2-2 {r} r_{\theta \theta}\right)\right)
\]
\[+\tan (\theta) \left(r_{t}^2-2 {r} r_{t t}\right) r_{\theta} \left(R^2 r_{t}^2+r_{\theta}^2+4 R^2 {r}\right)\Big)\]

\[\text{Ric}_{\alpha t}=0\]

\[\text{Ric}_{\theta \alpha}=0\]

\[\text{Ric}_{\phi \alpha}=0\]

\[\text{Ric}_{t \alpha}=0\]

\[\text{Ric}_{\alpha \alpha}=\frac{2 {r}}{\left(R^2 r_{t}^2+r_{\theta}^2+4 R^2 {r}\right)^2} \times\]
\[ \Big( 2 R^2 r_{t} r_{\theta} r_{\theta t}+R^2 r_{t}^2 \left(-r_{\theta \theta}+\tan (\theta) r_{\theta}+2 R^2\right)-4 R^2 {r} \left(R^2 r_{t t}+r_{\theta \theta}-\tan (\theta) r_{\theta}\right)\]
\[+r_{\theta}^2 \left(R^2 \left(-r_{t t}\right)+\tan (\theta) r_{\theta}+2 R^2\right) \Big)\]

The scalar curvature is then $S = \sum_{ij}g^{ij}\text{Ric}_{ij}$,

\[S = \frac{2}{\left(R^2 \left(4
    r+{r_t}^2\right)+{r_\theta}^2\right)^2} \times
\]
\[
 \Big(16 r^2 R^2-{r_\theta} \tan (\theta ) \left(4 r \left(R^2 ({r_{tt}}-2)+{r_{\theta \theta}}\right)+{r_t} \left(-4 R^2 {r_t}+{r_t} {r_{\theta \theta}}-2 {r_\theta}
   {r_{\theta t}}\right)+({r_{tt}}-4) {r_\theta}^2\right)+
\]
\[4 r \left(-2 R^4 {r_{tt}}+R^2 \left(2 {r_t}^2+({r_{tt}}-2) {r_{\theta \theta}}-{r_{\theta t}}^2\right)+{r_\theta}^2\right)+4 R^4
    {r_t}^2+
 \]
\[R^2 \left({r_t}^4-4 {r_t}^2 {r_{\theta \theta}}+8 {r_t} {r_\theta} {r_{\theta t}}-4 ({r_{tt}}-1) {r_\theta}^2\right)+{r_t}^2 {r_\theta}^2 \Big).\]

We want to find a function $r(\theta,t)$ satisfying the boundary conditions, for which this is positive.


\begin{figure}[ht!]
\centering
\includegraphics[width=150mm]{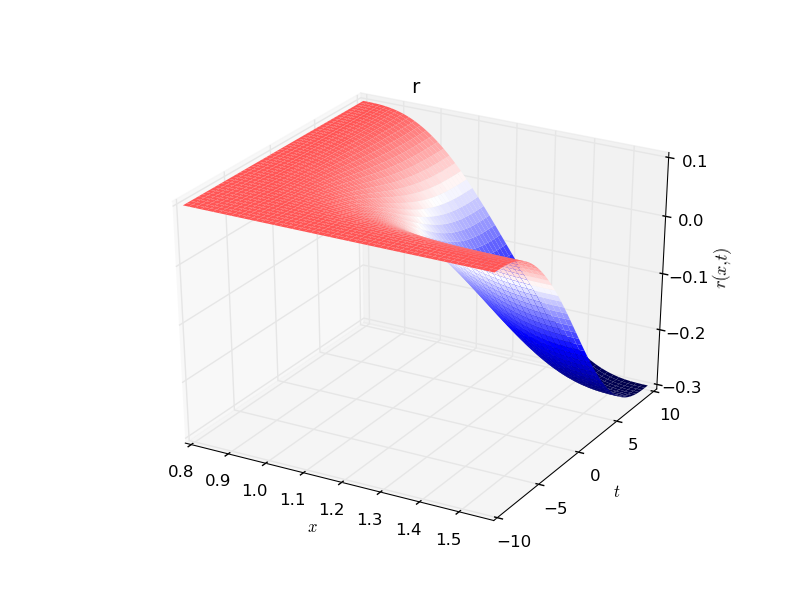}
\caption[A function used to construct the metrics of positive scalar curvature]{The function $r$, in terms of $\theta$ and $t$, where $T=10$ and $\theta_0=0.8$.}
\label{r_quintic}
\end{figure}

In the numerator of the expression above for scalar curvature, 
\[
 \Big(16 r^2 R^2-{r_\theta} \tan (\theta ) \left(4 r \left(R^2 ({r_{tt}}-2)+{r_{\theta \theta}}\right)+{r_t} \left(-4 R^2 {r_t}+{r_t} {r_{\theta \theta}}-2 {r_\theta}
   {r_{\theta t}}\right)+({r_{tt}}-4) {r_\theta}^2\right)+
\]
\[4 r \left(-2 R^4 {r_{tt}}+R^2 \left(2 {r_t}^2+({r_{tt}}-2) {r_{\theta \theta}}-{r_{\theta t}}^2\right)+{r_\theta}^2\right)+4 R^4
    {r_t}^2+
 \]
\[R^2 \left({r_t}^4-4 {r_t}^2 {r_{\theta \theta}}+8 {r_t} {r_\theta} {r_{\theta t}}-4 ({r_{tt}}-1) {r_\theta}^2\right)+{r_t}^2 {r_\theta}^2 \Big)\]
Let us consider scaling $R$ to $\alpha R$ and T to $\alpha T$. 

Then, the leading order term in $\alpha$ is 
\[N(\theta, t)=16 r^2R^2 +4 R^4 r_t^2+4R^2r_\theta^2-8R^4 r r_{tt} -8 R^2 r r_{\theta\theta}+8 r R^2 r_\theta \tan(\theta).\]

Consider the following claim:

\begin{claim}\label{leading_in_alpha_nonneg_part} 
For $R = 1$ and $T = 1$, where is a smooth function $r(\theta, t)$ satisfying the boundary conditions \ref{r_boundary_conditions} so that the quantity
\[-8R^4 r r_{tt} -8 R^2 r r_{\theta\theta}+8 r R^2 r_\theta \tan(\theta)\]
is non-negative for all $(\theta, t)$ in $[\theta_0, \pi/2] \times [-T, T]$.
\end{claim}

\begin{proof}[Proof that Claim \ref{leading_in_alpha_nonneg_part} implies the lemma]
First let us show that this claim suffices to show the lemma. Note that $N(\theta, t)$ is constant in $t$ for $t>T$ and for $t < -T$, and is constant in $\theta$ for $\theta \leq \theta_0$. Thus, its infimum is achieved for some $(\theta, t) \in [\theta_0, \pi/2] \times [-T, T]$. However, in this region, since $-8R^4 r r_{tt} -8 R^2 r r_{\theta\theta}+8 r R^2 r_\theta \tan(\theta) \geq 0$, we have that
\[N(\theta,t)  = 16 r^2R^2 +4 R^4 r_t^2+4R^2r_\theta^2-8R^4 r r_{tt} -8 R^2 r r_{\theta\theta}+8 r R^2 r_\theta \tan(\theta) \geq 6 r^2R^2 +4 R^4 r_t^2+4R^2r_\theta^2.\]
Since $r, r_t, r_\theta$ are not simultaneously $0$, for every $\theta, t$ in $[\theta_0, \pi/2] \times [-T, T]$, $N(\theta, t)$ is positive, so since its infinum is achieved for some $(\theta, t)$ in this region, the infimum of $N(\theta, t)$ is also positive. Say it is $>\delta$ for some $\delta>0$.

Then, scaling $R = \alpha$ and $T = \alpha$, we have that in the numerator of the scalar curvature, the leading order term in $\alpha$ is strictly positive, so for sufficiently large $\alpha$, $S$ is also strictly positive, as desired.

\end{proof}
The remainder of this proof is the proof of the claim. 

Let $\varepsilon_1$ be a real number such that $0 < \varepsilon_1 < \frac{1}{50}\(\frac{\pi}{2}-\theta_0\)$. 

Let $r_0$ be a positive number with $r_0<\frac{1}{2}$. Let us consider $r(\theta,t) = r_0-P(t)Q(\theta)$, where $P$ and $Q$ are smoothed step functions on $[-T,T]$ and $[\theta_0+\varepsilon_1,\pi/2-\varepsilon_1]$ respectively, that is they are a scale/translate of a step function $s(x)$ on $[0,1]$. 

We will also assume that $P$ and $P'$ are positive on $(-T,T)$ and $Q$ and $Q'$ are positive on $(\theta_0 +\varepsilon_1, \pi/2-\varepsilon_1)$, 
and $P''$ and $Q''$ are positive on $(-T,0)$ and $(\theta_0+\varepsilon_1,\frac{\theta_0+\pi/2}{2})$ respectively and negative on $(0,T)$ and $(\frac{\theta_0+\pi/2}{2}, \pi/2-\varepsilon_1)$, respectively. Moreover, we require that $P'$ and $Q'$ are symmetric about the middle of their supports.

That is, $P'$ and $Q'$ have to look like scales (and translates) of Figure \ref{Qx} and $P''$ and $Q''$ should look like scales (and translates) of Figure \ref{Qxx}. 

\begin{figure}[ht!]
\centering
\includegraphics[width=80mm]{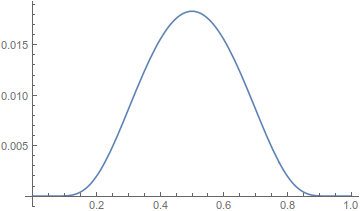}
\caption[]{$Q'$}
\label{Qx}
\end{figure}

\begin{figure}[ht!]
\centering
\includegraphics[width=80mm]{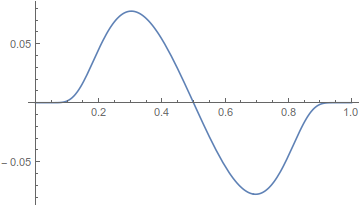}
\caption[]{$Q''$}
\label{Qxx}
\end{figure}

We will moreover require that $\lim_{x \to 0} \frac{s(x)}{s'(x) \cdot x^2}>0$ and  $\lim_{x \to 0} \frac{s'(x)}{s''(x) \cdot x^2}>0$ (meaning the limits exist and are positive). These conditions all hold for
\[s(x) = \text{constant} \cdot \int_0^x e^{-1/y-1/(1-y)}dy.\]

Let us show that such an $r(\theta, t) = r_0-P(t)Q(\theta)$ satisfies the boundary conditions \ref{r_boundary_conditions}:

\begin{enumerate}
\item Since $P$ is constant in $t$ for $t<-T$ and for $t>T$, $r$ is as well.
\item When $t\leq -T$ or $\theta \leq \theta_0$, we have $P(t)Q(\theta)=0$, so $r(\theta, t) = r_0$.
\item Since $Q$ is a step function on $[\theta_0+\varepsilon_1, \pi/2-\varepsilon_1]$, we have that it is constant in $\theta$ for $\theta <\theta_0+\varepsilon_1$, so $r$ is as well.
\item For $t>T$, $r(\pi/2, t) = r_0-P(t)Q(\pi/2) = r_0-1<0$. 
\item Since $P$ and $Q$ are non-decreasing $t$ and $\theta$ respectively, we have that $r$ is non-increasing in $\theta$ and $t$.
\item Note that $r_\theta = -Q'(\theta)$ only vanishes where $Q$ is constant, ie where $Q$ takes values $0$ or $1$, and $r_t = -P'(t)$ only vanishes where $P$ is constant, ie, where $P$ takes values $0$ or $1$. Thus, for $r_\theta$ and $r_t$ both to vanish, we would have $r(\theta, t) = r_0-P(t)Q(\theta)$ must be either $r_0$ or $r_0-1$, neither of which are $0$, since $0<r_0<1/2$. 
\end{enumerate}

We wish to show that for this $r$, we have that for $(\theta, t) \in [\theta_0, \pi/2] \times [-T, T]$, we have $-8R^4 r r_{tt} -8 R^2 r r_{\theta\theta}+8 r R^2 r_\theta \tan(\theta) \geq 0$.

But note that for $\theta \leq \theta_0+\varepsilon_1$ and $\theta \geq \pi/2 - \varepsilon_1$, $r$ is constant in $\theta$, so it suffices to show that $-8R^4 r r_{tt} -8 R^2 r r_{\theta\theta}+8 r R^2 r_\theta \tan(\theta) \geq 0$ when $r \geq 0$ and $(\theta, t) \in [\theta_0+\varepsilon_1, \pi/2-\varepsilon_1] \times [-T, T]$

Note that we are allowed to decrease $r_0$ as long as it remains positive. Hence, it suffices to show that there exists $\varepsilon$ such that
\[-8r_{tt} -8 r_{\theta\theta}+8 r_\theta \tan(\theta) \geq 0\]
whenever $\theta < \theta_0+\varepsilon_1+\varepsilon$ or $t<-T+\varepsilon$, because if we had this, then we can choose $r_0$ so that $r(\theta_0+\varepsilon_1+\varepsilon,-T+\varepsilon)=0$, so that $r \geq 0$ means either $\theta < \theta_0+\varepsilon_1+\varepsilon$ or $t<-T+\varepsilon$.

\begin{figure}[ht!]
\centering
\includegraphics[width=100mm]{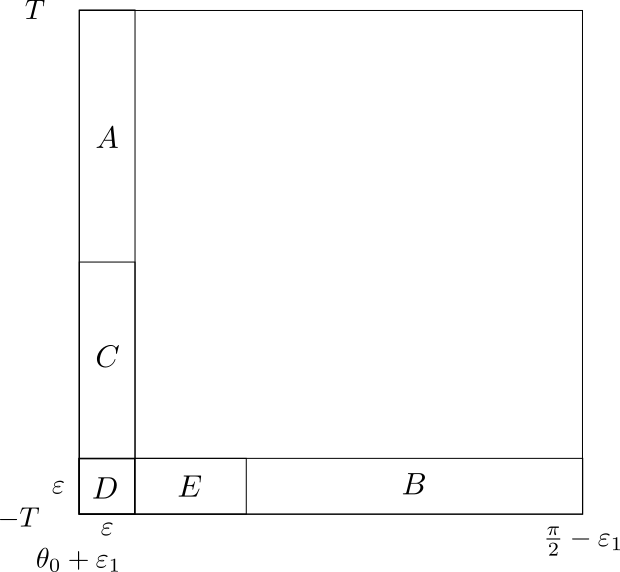}
\caption[]{}
\label{blocks}
\end{figure}

Thus, it suffices to show that there is $\varepsilon>0$ such that on each of the blocks $A,B,C,D,E$ in Figure \ref{blocks}, $-8r_{tt} -8 r_{\theta\theta}+8 r_\theta \tan(\theta) \geq 0$, where the boundary between $A$ and $C$ is $t=0$ and the boundary between $E$ and $B$ is $\theta=\frac{3\theta_0+\pi/2}{4}$. 

We can rewrite the desired inequality, $-8r_{tt} -8 r_{\theta\theta}+8 r_\theta \tan(\theta) \geq 0$ as
\[Q(\theta)P''(t)+Q''(\theta)P(t)-Q'(\theta)P(t)\tan(\theta) \geq 0.\]

In \textbf{region $A$}, note that $Q''(\theta)>0$ and $P(t) \geq 1/2$, so that $Q''(\theta)P(t)$ is positive. We will show that for sufficiently small $\varepsilon$, this term dominates the other two in region $A$, both of which are negative, since for $0 \leq t \leq T$ and all values of $\theta$, $P''(t) \leq 0$ and $Q'(\theta) \geq 0$. 

Note that for any $M>0$, for $\varepsilon$ sufficiently small $Q''(\theta) > 2MQ'(\theta)+ 2MQ(\theta)$, where both terms on the right are also positive. This follows from the conditions $\lim_{x \to 0} \frac{s(x)}{s'(x) \cdot x^2}>0$ and  $\lim_{x \to 0} \frac{s'(x)}{s''(x) \cdot x^2}>0$ for $s$. This means that $Q''(\theta)P(t)> MQ'(\theta) + MQ(\theta)$. 

Also for $\varepsilon< \frac{1}{2}(\pi/2-\theta_0-\varepsilon_1$, $ \tan(\theta) <\tan\(\frac{\pi/2+\theta_0}{2}\)$ is bounded above in region $A$, and $P(t) \geq 1/2$ and $P''(t)$ and is bounded.

Thus, choosing $M > \tan\(\frac{\pi/2+\theta_0}{2}\) + \sup(\{P''(t) | t \in [-T, T]\})$, we get 
\[Q(\theta)P''(t)+Q''(\theta)P(t)-Q'(\theta)P(t)\tan(\theta) > Q(\theta)P''(t)+MQ(\theta)+ MQ'(\theta)P(t)-Q'(\theta)P(t)\tan(\theta)\]
and by our choice of $M$, it is easy to see that the latter is non-negative.

In \textbf{region $B$}, we have that for any $M>0$, for sufficiently small $\varepsilon$, $P''(t) > 2M P(t)$. Moreover, $Q(\theta) \geq s(1/8)>0$. Also $Q''(\theta)$ and $Q'(\theta)\cdot \tan(\theta)$ are bounded. Thus, choosing $M>\frac{1}{s(1/8)}\(\sup\{Q'(\theta) \tan(\theta) | \theta \in [\theta_0, \pi/2]\} + \sup\{|Q''(\theta)|  | \theta \in [\theta_0, \pi/2]\}\)$, we get that

\[Q(\theta)P''(t)+Q''(\theta)P(t)-Q'(\theta)P(t)\tan(\theta) \geq 2Q(\theta) MP(t)+Q''(\theta)P(t)-Q'(\theta)P(t)\tan(\theta) \]
which is again non-negative by our choice of $M$.

For \textbf{regions $C$ and $D$}, both the $Q(\theta)P''(t)$ and $Q''(t)P(\theta)$ terms are non-negative. Also for any $M>0$ we can choose $\varepsilon$ sufficiently small that in these regions, $Q''(\theta)>2MQ'(\theta)$. Note that $\tan(\theta)$ is bounded above in regions $C$ and $D$, by $\tan(\theta)<\tan\(\frac{\theta_0 + \pi/2}{2}\)$. Picking $M>\tan\(\frac{\theta_0 + \pi/2}{2}\)$, we get
\[Q(\theta)P''(t)+Q''(\theta)P(t)-Q'(\theta)P(t)\tan(\theta)  \geq 2\tan(\theta)Q'(\theta)P(t)-Q'(\theta)P(t)\tan(\theta)\geq 0,\]
as desired.

In \textbf{region $E$}, note that $P''(t)$ and $Q''(\theta) \geq 0$.

Also note that there is some positive $C$ such that for $\theta<\theta_1$, $Q(\theta) Q''(\theta) \geq C (Q'(\theta))^2$. This is because $\lim_{\theta \to 0} \frac{Q(\theta) Q''(\theta)}{(Q'(\theta))^2}>0$ so $\frac{Q(\theta) Q''(\theta)}{(Q'(\theta))^2} >0$ on $[\theta_0,\theta=\frac{3\theta_0+\pi/2}{4}]$ and we can let $C$ be the infimum of $\frac{Q(\theta) Q''(\theta)}{(Q'(\theta))^2} $ in this region, which is achieved, and therefore positive.

Similarly, on $E$, for sufficiently small positive $\varepsilon$ and $C$, $\frac{P(t)P''(t)}{(P'(t))^2} \geq C$. Then
\[P''(t)Q(\theta)+Q''(\theta)P(t) \geq 2 \sqrt{P''(t)Q(\theta)Q''(\theta)P(t)} \geq 2C|P'(t)Q'(\theta)|,\]
and it suffices to show that $2C|P'(t)Q'(\theta)| -Q'(\theta)P(t)\tan(\theta) \geq 0$. But for any $M>0$, for sufficiently small $\varepsilon$, we have $P'(t) >M P(t)$, and in region $E$, $\tan(\theta)$ is bounded, so the inequality holds, as desired.

\end{proof}

Now let us use the above construction to show the following lemma; it is the same as the above, but no longer allows us to vary $R$.

\begin{lemma} 
For any $R$ and $\theta_0<\pi/2$, there is a metric on $D^2 \times D^2$ as a cobordism $D^2 \times S^1 \to S^1 \times D^2$ that agrees with $D^2_{R,\theta_0} \times S^1$ on the $-\infty$ end and is constant in time near the boundary $S^1 \times S^1 \subset D^2 \times S^1$. 
\end{lemma}
\begin{proof} Consider the metric in the proof of the lemma above:
\[
g_{\mu\nu}=\left(
\begin{array}{cccc}
 R^2+\frac{1}{4 r}{r_\theta}^2 & 0 & \frac{1}{4 r} {r_t} {r_\theta}& 0 \\
 0 & R^2 \cos ^2(\theta ) & 0 & 0 \\
 \frac{1}{4 r} {r_t} {r_\theta}& 0 & \frac{1}{4 r}{r_t}^2+1 & 0 \\
 0 & 0 & 0 & r \\
\end{array}
\right)
\]
Choose some $R_0$ and $r$ such that this has positive scalar curvature; the previous lemma shows that such $r$ exists for $R_0$ sufficiently large, and the $r(\theta,t)$ constructed has form $r_0-P(t)Q(\theta)$ where $P$ and $Q$ are smoothed step functions on $[-T,T]$ and $[\theta_0,\pi/2]$.

Consider replacing the metric with
\[
g_{\mu\nu}=\left(
\begin{array}{cccc}
 R^2+\frac{1}{4 r}{r_\theta}^2 & 0 & \frac{1}{4 r} {r_t} {r_\theta}& 0 \\
 0 & R^2 \cos ^2(\theta ) & 0 & 0 \\
 \frac{1}{4 r} {r_t} {r_\theta}& 0 & \frac{1}{4 r}{r_t}^2+1 & 0 \\
 0 & 0 & 0 & cr \\
\end{array}
\right)
\]
for some constant $c$. This does not affect the scalar curvature of the metric.

For $t<-T$, the metric now looks like
\[
g_{\mu\nu}=\left(
\begin{array}{cccc}
 R^2 & 0 & 0& 0 \\
 0 & R^2 \cos ^2(\theta ) & 0 & 0 \\
 0 & 0 & 1 & 0 \\
 0 & 0 & 0 & c r_0 \\
\end{array}
\right)
\]
which is the same as the previous one, but we have scaled the $S^1$ in $D^2_R \times S^1$. For general $\theta_0$, we may now simply pick large $R$ and small $r_0$ such that there is $r(\theta,t)$ as above, so the cobordism has positive scalar curvature, then scale the entire picture (the $S^2$ part, the $S^1$ part, and time) down to make $R$ match the original value we wanted, and then scale the $r$ coordinate back so that $\sqrt{r}$ matches the desired $S^1$ radius by varying $c$. (Note that for convenience in the proof of the previous lemma, we have chosen $r$ to be the square root of the $S^1$ radius.)
\end{proof}

\subsection{when the $S^2$ has radius 1}

Repackaging the metric in the previous section in a way that makes it easier for later computations, but slightly less intuitive, we may consider on $S^1 \times S^2$, the metric
\[
g_{\mu\nu}=\left(
\begin{array}{cccc}
 1+\frac{r_\theta^2}{4 rR^2} & 0 & \frac{r_tr_\theta}{4 rR^2}& 0 \\
 0 & \cos ^2(\theta) & 0 & 0 \\
\frac{r_tr_\theta}{4 rR^2}& 0 & 1+\frac{r_t^2}{4 rR^2} & 0 \\
 0 & 0 & 0 & r \\
\end{array}
\right)
\]
where $r = 1-MP(t)Q(\theta)$.

This is the metric on $\rr^4$ inherited from the map $\rr^4 \to \rr^6$ given by
\[(\theta, \phi, t, a) \mapsto (\cos(\theta)\cos(\phi), \cos(\theta)\sin(\phi), \sin(\theta), \sqrt{r(\theta,t)}/R \cos(Ra) ,\sqrt{r(\theta,t)}/R \sin(Ra), t),\]
so it is a well defined metric.

When $t<-T$ or $\theta<\theta_0$, the metric is 
\[
\left(
\begin{array}{cccc}
 1 & 0 &0 & 0 \\
 0 & \cos ^2(\theta) & 0 & 0 \\
0 & 0 & 1 & 0 \\
 0 & 0 & 0 & 1 \\
\end{array}
\right).
\]

This has scalar curvature 
\[\frac{2}{\left(4 R^2 r+{r_{t}}^2+{r_\theta}^2\right)^2} \times\]
\[\left(16 R^4 r^2-4 R^2 r \left({r_{tt}} \left(2 R^2+{r_\theta} \tan ({\theta})-{r_{\theta\theta}}\right)-2 R^2 {r_\theta} \tan ({\theta})+2 R^2 {r_{\theta\theta}}-2
   {r_{t}}^2-{r_\theta}^2+{r_\theta} {r_{\theta\theta}} \tan ({\theta})+{r_{\theta t}}^2\right) \right.\]
\[\left.+{r_{t}}^2 \left({r_\theta} \left(4 R^2-{r_{\theta\theta}}\right) \tan ({\theta})+4 R^2
   \left(R^2-{r_{\theta\theta}}\right)+{r_\theta}^2\right)+2 {r_{t}} {r_\theta} {r_{\theta t}} \left(4 R^2+{r_\theta} \tan ({\theta})\right)\right.\]
\[\left.+{r_\theta}^2 \left(4 \left(R^4+R^2 {r_\theta} \tan
   ({\theta})\right)-{r_{tt}} \left(4 R^2+{r_\theta} \tan ({\theta})\right)\right)+{r_{t}}^4\)\]
the leading order term of the numerator in $R$ is
\[16 {r}^2-8 {r} {r_{tt}}+8 {r} {r_{\theta}} \tan (\theta)-8 {r} {r_{\theta\theta}}+4 {r_{t}}^2+4 {r_{\theta}}^2\]
And, as in the previous section, let $r(\theta,t) = 1-MP(t)Q(\theta)$, where $P$ is the smoothed step function on $[-1,1]$, and $Q$ is the smoothed step function on $[\theta_0, \pi/2]$. When $M$ is sufficiently large, we can make the leading term in $R$ positive, so that when $R$ is sufficiently large, the scalar curvature is positive whenever $r \geq 0$.

This gives us a positive scalar curvature on $S^2 \times S^1$ with a handle attached such that for $t<-1$, the metric is \[
g_{\mu\nu}=\left(
\begin{array}{cccc}
 1 & 0 & 0& 0 \\
 0 & \cos ^2(\theta) & 0 & 0 \\
0& 0 & 1 & 0 \\
 0 & 0 & 0 & 1 \\
\end{array}
\right).
\]
By taking $\theta_0 \to \pi/2$, we can push the handle attachment arbitrarily close to the pole.

Note that the logic here is this: First find $r$ of the form $1-P(t)Q(\theta)M$ such that $-r_{tt}+r_\theta \tan(\theta) -r_{\theta \theta} \geq 0$, where it is equal to zero only on the boundary, where $\theta = \theta_0$ or $t=-1$. Then, for sufficiently large $R$, the scalar curvature expression above is positive. Then this is the scalar curvature for a metric in coordinates $(\theta, \phi, t, \alpha/R)$, where the $\alpha$ has been divided by $R$ because we scaled the last coordinate.

\bibliography{main}
\bibliographystyle{plain}

\end{document}